\documentclass[11pt]{amsart}
\usepackage{amssymb,amscd,amsxtra,calc}
\usepackage{cmmib57}
\usepackage{graphicx}
\usepackage{amssymb,amsmath,amsfonts,mathrsfs,enumerate,xcolor}
\usepackage{amsthm}
\usepackage{soul}
\usepackage[citecolor=Navy]{hyperref}
\usepackage[ruled,vlined]{algorithm2e}

\usepackage{lipsum}
\usepackage{amsfonts}
\usepackage{graphicx}
\usepackage{epstopdf}
\usepackage{algorithmic}
\usepackage{float}
\restylefloat{table}
\usepackage{placeins}
\usepackage{array}
\ifpdf
  \DeclareGraphicsExtensions{.eps,.pdf,.png,.jpg}
\else
  \DeclareGraphicsExtensions{.eps}
\fi

\usepackage{subcaption}

\theoremstyle{plain}
    \newtheorem{thm}{Theorem}[section]

    \newtheorem{theorem}[thm]{Theorem}

\theoremstyle{definition}

\theoremstyle{remark}

\newcommand{\authorfootnotes}{\renewcommand\thefootnote{\@fnsymbol\c@footnote}}

\usepackage{caption} 
 \title[New Q-Newton's method Backtracking: Generalisations and improvements]{Generalisations and improvements of New Q-Newton's method Backtracking}
 \author{Tuyen Trung Truong}
   \address{Department of Mathematics, University of Oslo, Blindern 0851 Oslo, Norway}
  \email{tuyentt@math.uio.no}
    \date{\today}
   \subjclass[2010]{}

\begin{document}
\maketitle
%{\centering\footnotesize To my daughter on her birthday occasion\par}

\begin{abstract}

In this paper, we propose a general framework for the algorithm New Q-Newton's method Backtracking, developed in the author's previous work. For a symmetric, square real matrix $A$, we define $minsp(A):=\min _{||e||=1} ||Ae||$. Given a $C^2$ cost function $f:\mathbb{R}^m\rightarrow \mathbb{R}$ and a  real number $0<\tau $, as well as  $m+1$ fixed real numbers $\delta _0,\ldots ,\delta _m$, we define for each $x\in \mathbb{R}^m$ with $\nabla f(x)\not= 0$ the following quantities: 

$\kappa :=\min _{i\not= j}|\delta _i-\delta _j|$;

$A(x):=\nabla ^2f(x)+\delta ||\nabla f(x)||^{\tau}Id$, where $\delta$ is the first element in the sequence $\{\delta _0,\ldots ,\delta _m\}$ for which $minsp(A(x))\geq \kappa ||\nabla f(x)||^{\tau}$;

$e_1(x),\ldots ,e_m(x)$  are an orthonormal basis of $\mathbb{R}^m$, chosen appropriately;

$w(x)=$ the step direction, given by the formula: 
$$w(x)=\sum _{i=1}^m\frac{<\nabla f(x),e_i(x)>}{||A(x)e_i(x)||}e_i(x);$$
(we can also normalise by $w(x)/\max \{1,||w(x)||\}$ when needed)

$\gamma (x)>0$ learning rate chosen by Backtracking line search so that Armijo's condition is satisfied: 
$$f(x-\gamma (x)w(x))-f(x)\leq -\frac{1}{3}\gamma (x)<\nabla f(x),w(x)>.$$

The update rule for our algorithm is $x\mapsto H(x)=x-\gamma (x)w(x)$. 

In New Q-Newton's method Backtracking, the choices are $\tau =1+\alpha >1$ and $e_1(x),\ldots ,e_m(x)$'s are eigenvectors of $\nabla ^2f(x)$. In this paper, we allow more flexibility and generality, for example $\tau$ can be chosen to be $<1$ or $e_1(x),\ldots ,e_m(x)$'s are not necessarily eigenvectors of $\nabla ^2f(x)$. 

New Q-Newton's method Backtracking (as well as Backtracking gradient descent) is a special case, and some versions have flavours of quasi-Newton's methods. Several versions allow good theoretical guarantees. An application to solving systems of polynomial equations is given.

\end{abstract}

\section{Introduction} Let $f:\mathbb{R}^m\rightarrow \mathbb{R}$ be a $C^2$ function, with gradient $\nabla f(x)$ and Hessian $\nabla ^2f(x)$.   Newton's method $x_{n+1}=x_n-(\nabla ^2f(x_n))^{-1}\nabla f(x_n)$ (if the Hessian is {\bf invertible}) is a very popular iterative optimization method. It seems that every month there is at least one paper about this subject appears on arXiv. One attractive feature of this method is that if it {\bf converges} then it usually converges very fast, with the rate of convergence being quadratic, which is generally faster than that of Gradient descent (GD) methods.  We recall that if $\{x_n\}\subset \mathbb{R}^m$  converges to $x_{\infty}$, and so that $||x_{n+1}-x_{\infty}||=O(||x_n-x_{\infty}||^{\epsilon})$, then $\epsilon$ is the rate of convergence of the given sequence. If $\epsilon =1$ then we have linear rate of convergence, while if $\epsilon=2$ then we have quadratic rate of convergence.  

However, there is no guarantee that Newton's method will converge, and it is problematic near points where the Hessian is not invertible. Moreover, it cannot avoid saddle points. Recall that a {\bf saddle point} is a point $x^*$ which is a non-degenerate critical point of $f$ (that is $\nabla f(x^*)=0$ and $\nabla ^2f(x^*)$ is invertible) so that the Hessian has at least one {\bf negative eigenvalue}. Saddle points are problematic in large scale optimization (such as those appearing in Deep Neural Networks, for which the dimensions could easily be millions or billions), see \cite{bray-dean, dauphin-pascanu-gulcehre-cho-ganguli-bengjo}. 

There are many modifications of Newton's method. The intended readers can see an overview in \cite{truong-etal}. This paper concerns only the versions recently developed in

The author's joint paper \cite{truong-etal} defined a new modification of Newton's method which is easy to implement, while can avoid saddle points and as fast as Newton's method. It is recalled in Algorithm \ref{table:alg}. Here we explain notations used in the description. Let $A:\mathbb{R}^m\rightarrow \mathbb{R}^m$ be an invertible {\bf symmetric} square matrix. In particular, it is diagonalisable.  Let $V^{+}$ be the vector space generated by eigenvectors of positive eigenvalues of $A$, and $V^{-}$ the vector space generated by eigenvectors of negative eigenvalues of $A$. Then $pr_{A,+}$ is the orthogonal projection from $\mathbb{R}^m$ to $V^+$, and  $pr_{A,-}$ is the orthogonal projection from $\mathbb{R}^m$ to $V^-$. As usual, $Id$ means the $m\times m$ identity matrix.

\medskip
{\color{blue}
 \begin{algorithm}[H]
\SetAlgoLined
\KwResult{Find a critical point of $f:\mathbb{R}^m\rightarrow \mathbb{R}$}
Given: $\{\delta_0,\delta_1,\ldots, \delta_{m}\}\subset \mathbb{R}$\ (chosen {\bf randomly}) and $\alpha >0$;\\
Initialization: $x_0\in \mathbb{R}^m$\;
 \For{$k=0,1,2\ldots$}{ 
    $j=0$\\
    \If{$\|\nabla f(x_k)\|\neq 0$}{
   \While{$\det(\nabla^2f(x_k)+\delta_j \|\nabla f(x_k)\|^{1+\alpha}Id)=0$}{$j=j+1$}}

$A_k:=\nabla^2f(x_k)+\delta_j \|\nabla f(x_k)\|^{1+\alpha}Id$\\
$v_k:=A_k^{-1}\nabla f(x_k)=pr_{A_k,+}(v_k)+pr_{A_k,-}(v_k)$\\
$w_k:=pr_{A_k,+}(v_k)-pr_{A_k,-}(v_k)$\\
$x_{k+1}:=x_k-w_k$
   }
  \caption{New Q-Newton's method} \label{table:alg}
\end{algorithm}
}
\medskip
 
 Experimentally, on small scale problems, New Q-Newton's method works quite competitive against the methods mentioned above, see \cite{truong-etal}.  However, while it can avoid saddle points, it does not have a descent property, and an open question in \cite{truong-etal} is whether it has good convergence guarantees. More recently, in \cite{truong2021}, the author defined  New Q-Newton's method Backtracking, see Algorithm  \ref{table:alg0} below, which resolves this convergence guarantee issue. It incorporates Backtracking line search into New Q-Newton's method. Its use of hyperparameters is more sophisticated than that of New Q-Newton's method, and we need some notations.  For a symmetric, square real matrix $A$, we define: 
  
  $sp(A)=$ the maximum among $|\lambda |$'s, where $\lambda  $ runs in the set of eigenvalues of $A$, this is usually called the spectral radius in the Linear Algebra literature;
  
  and 
  
  $minsp(A)=$ the minimum among $|\lambda |$'s, where $\lambda  $ runs in the set of eigenvalues of $A$, this number is non-zero precisely when $A$ is invertible.
  
 One can easily check the following more familiar formulas: $sp(A)=\max _{||e||=1}||Ae||$ and $minsp(A)=\min _{||e||=1}||Ae||$, using for example the fact that $A$ is diagonalisable.  

\medskip
{\color{blue}
 \begin{algorithm}[H]
\SetAlgoLined
\KwResult{Find a critical point of $f:\mathbb{R}^m\rightarrow \mathbb{R}$}
Given: $\{\delta_0,\delta_1,\ldots, \delta_{m}\} \subset \mathbb{R}$\ (chosen {\bf randomly}), $0<\alpha $ and $0<\gamma _0<1$;\\
Initialization: $x_0\in \mathbb{R}^m$\;
$\kappa:=\frac{1}{2}\min _{i\not=j}|\delta _i-\delta _j|$;\\
 \For{$k=0,1,2\ldots$}{ 
    $j=0$\\
  \If{$\|\nabla f(x_k)\|\neq 0$}{
   \While{$minsp(\nabla^2f(x_k)+\delta_j \|\nabla f(x_k)\|^{1+\alpha}Id)<\kappa  \|\nabla f(x_k)\|^{1+\alpha}$}{$j=j+1$}}
  
 $A_k:=\nabla^2f(x_k)+\delta_j \|\nabla f(x_k)\|^{1+\alpha}Id$\\
$v_k:=A_k^{-1}\nabla f(x_k)=pr_{A_k,+}(v_k)+pr_{A_k,-}(v_k)$\\
$w_k:=pr_{A_k,+}(v_k)-pr_{A_k,-}(v_k)$\\
$\widehat{w_k}:=w_k/\max \{1,||w_k||\}$\\
$\gamma :=\gamma _0$\\
 \If{$\|\nabla f(x_k)\|\neq 0$}{
   \While{$f(x_k-\gamma \widehat{w_k})-f(x_k)>-\gamma <\widehat{w_k},\nabla f(x_k)>/3$}{$\gamma =\gamma /3$}}

$x_{k+1}:=x_k-\gamma \widehat{w_k}$
   }
  \caption{New Q-Newton's method Backtracking} \label{table:alg0}
\end{algorithm}
}
\medskip
 
There are two main differences between New Q-Newton's method  and New Q-Newton's method Backtracking. First, in the former we only need $\det(A_k)\not= 0$, while in the latter we need the eigenvalues of $A_k$ to be  "sufficiently large". Second, the former has no line search component, while the latter has. Note that the $\widehat{w_k}$ in Algorithm \ref{table:alg0} satisfies $||\widehat{w_k}||\leq 1$, and if $||\widehat{w_k}||<1$ then $\widehat{w_k}=w_k$. On the other hand, if one wants to keep closer to New Q-Newton's method, one can apply line search directly to $w_k$, and obtain a variant which will be named New Q-Newton's method Backtracking S, see Section \ref{Section2} for its description and theoretical guarantees. 
 
In \cite{truong2021}, there are also some other versions of New Q-Newton's method Backtracking, both theoretically and practically. For simplicity, here we concerns only the above version. It is shown in \cite{truong2021} that any cluster point of a sequence $\{x_n\}$ constructed by New Q-Newton's method is a critical point of $f$. Moreover, in case $f$ is a Morse function, then it is shown that if the initial point $x_0$ is random, then either $\lim _{n\rightarrow\infty}||x_n||=\infty$ or $\{x_n\}$ converges to a local minimum with quadratic rate of convergence. In an updated version of \cite{truong-etal}, this is used to show theoretically that one can quickly find roots of univariate meromorphic functions, and experiments show that even New Q-Newton's method works well for this task. 

This suggests that New Q-Newton's method Backtracking (or New Q-Newton's method) also has good theoretical guarantees when the cost function is real analytic, or more generally if both $f$ and $\nabla f$ satisfy Lojasiewicz gradient inequality. Besides Morse functions, the latter type of cost functions are met in many realistic applications, for example in robotics or Deep Neural Networks. Experimentally, this is confirmed on various examples, including non-smooth ones, see \cite{truong2021}. In this paper, we provide more theoretical support to this, by showing that a modification of New Q-Newton's method Backtracking indeed has this good theoretical guarantee. 

We start with a generalised framework, see Algorithm \ref{table:algG}. 

\medskip
{\color{blue}
 \begin{algorithm}[H]
\SetAlgoLined
\KwResult{Find a critical point of $f:\mathbb{R}^m\rightarrow \mathbb{R}$}
Given: $\{\delta_0,\delta_1,\ldots, \delta_{m}\} \subset \mathbb{R}$\ (chosen {\bf randomly}), $0<\tau $ and $0<\gamma _0<1$;\\
Initialization: $x_0\in \mathbb{R}^m$\;
$\kappa:=\frac{1}{2}\min _{i\not=j}|\delta _i-\delta _j|$;\\
 \For{$k=0,1,2\ldots$}{ 
    $j=0$\\
  \If{$\|\nabla f(x_k)\|\neq 0$}{
   \While{$minsp(\nabla^2f(x_k)+\delta_j \|\nabla f(x_k)\|^{\tau}Id)<\kappa  \|\nabla f(x_k)\|^{\tau}$}{$j=j+1$}}
  
 $A_k:=\nabla^2f(x_k)+\delta_j \|\nabla f(x_k)\|^{\tau}Id$\\
$e_{1,k},\ldots ,e_{m,k}$ an appropriately chosen orthonormal basis for $\mathbb{R}^m$\\
$w_k:=\sum _{i=1}^m\frac{<\nabla f(x_k),e_{i,k}>}{||A_k.e_{i,k}||}e_{i,k}$\\
$\widehat{w_k}:=w_k/\max \{1,||w_k||\}$\\
$\gamma :=\gamma _0$\\
 \If{$\|\nabla f(x_k)\|\neq 0$}{
   \While{$f(x_k-\gamma \widehat{w_k})-f(x_k)>-\gamma <\widehat{w_k},\nabla f(x_k)>/3$}{$\gamma =\gamma /3$}}

$x_{k+1}:=x_k-\gamma \widehat{w_k}$
   }
  \caption{New Q-Newton's method Backtracking G} \label{table:algG}
\end{algorithm}
}
\medskip

There are two main changes from New Q-Newton's method Backtracking, which bring in more generalities and flexibilities as well as more convergence guarantees. First, the choice of the exponent $\mu$ is now allowed to be $\leq 1$, which allows more convergence guarantees.  Second, we don't need to choose $e_{i,k}$ to be eigenvectors of $\nabla ^2f(x_k)$.

We now check that the new algorithm includes both New Q-Newton's method Backtracking and  Backtracking gradient descent as special cases. 

{\bf New Q-Newton's method Backtracking:} (see \cite{truong2021}) We choose $\tau =1+\alpha >1$ and $e_{1,k},\ldots ,e_{m,k}$ eigenvectors of $A_k$. Hence, if $\lambda _i$ is the corresponding eigenvalue of $e_{i,k}$ of $A_k$, then 
\begin{eqnarray*}
w_k=\sum _{i=1}^m\frac{<\nabla f(x_k),e_{i,k}>}{|\lambda _i|}e_{i,k}. 
\end{eqnarray*}

{\bf Backtracking gradient descent:} We choose $e_{1,k}=\nabla f(x_k)/||\nabla f(x_k)||$. Then, by definition, we have $<\nabla f(x_k),e_{l.k}>=0$ for $l=2,\ldots ,m$. Thus in this case: 
\begin{eqnarray*}
w_k=\frac{||\nabla f(x_k)||}{||\nabla ^2f(x_k).\nabla f(x_k)||}\nabla f(x_k). 
\end{eqnarray*}
Hence, we recover (rescaled) Backtracking gradient descent. 

Roughly speaking, near a non-degenerate critical point, then the best learning rate one can choose for Armijo's condition is about $1/||\nabla ^2f(x_k)||$ (see e.g. \cite{truong-nguyen2}). The above version gives a more precise value of the learning rate to start with. Compare also \cite{truong, truongnew, truong2021}.

We now define some new special cases of New Q-Newton's method Backtracking G.  We fix $e_1,\ldots ,e_m$ an orthonormal basis for $\mathbb{R}^m$. 

{\bf New Q-Newton's method Backtracking G1:} We choose $0<\tau \leq 1$, and choose $e_{1,k},\ldots ,e_{m,k}$ to consist of eigenvectors of $A_k$. This is almost the same as New Q-Newton's Backtracking, except that there $\tau =1+\alpha >1$. 

The benefit is that this allows one to prove new theoretical guarantees for real analytic cost functions, while the associated dynamical system is still $C^1$ near non-degenerate critical points and hence it still can avoid saddle points. 

{\bf New Q-Newton's method Backtracking G2:} We choose $0<\tau \leq 1$ and choose $e_{i,k}=e_i$ for all $i=1,\ldots ,m$. 

This version is probably the less computationally expensive, and hence has some favours of quasi-Newton's methods. 

{\bf New Q-Newton's method Backtracking G3:} We choose $e_{i,k}=e_i$ for all $i=1,\ldots ,m$ as follows: Near a non-degenerate critical point $x^*$, we choose $e_{j}(x)$ ($j=1,\ldots  ,m$) consistently by only one of the following schemes: 

i) $e_{j}(x)$'s are eigenvectors of $\nabla ^2f(x)$;

or 

ii) $e_{j}(x)$'s are $C^1$ functions for all $j=1,\ldots ,m$. 

Both New Q-Newton's method Backtracking (G1) (where we always follow rule i) and G2 (where we always follow rule ii) are special cases of New Q-Newton's method G3. Here is one way to combine both of them: 

{\bf New Q-Newton's method Backtracking G4:} If $minsp(A_k)\geq \kappa ||\nabla f(x_k)||^{1/2}$, then we choose $e_{1,k},\ldots ,e_{m,k}$ to be eigenvectors of $\nabla ^2f(x_k)$. Otherwise, we choose $e_{1.k},\ldots ,e_{m,k}$ to be $e_1,\ldots ,e_k$. 

In the remaining of this section, we present the good theoretical guarantees of the new algorithms.  We recall that if $\{x_n\}\subset \mathbb{R}^m$ has a subsequence $\{x_{n_k}\}$ converging to a point $x_{\infty}$, then $x_{\infty}$ is a {\bf cluster point} of $\{x_n\}$. By definition, the sequence $\{x_n\}$ is convergent if it is bounded and has exactly one cluster point.  We note that the below results generalise those known for Backtracking gradient descent and New Q-Newton's method Backtracking. 

\begin{theorem}
Let $f:\mathbb{R}^m\rightarrow \mathbb{R}$ be a $C^3$ function. Let $\{x_n\}$ be a sequence constructed by the New Q-Newton's method Backtracking G. 

0) (Descent property) $f(x_{n+1})\leq f(x_n)$ for all n. 

1) If $x_{\infty}$ is a {\bf cluster point} of $\{x_n\}$, then $\nabla f(x_{\infty})=0$. That is, $x_{\infty}$ is a {\bf critical point} of $f$.

2) If $0<\tau <1$: If $f$ has at most countably many critical points, then either $\lim _{n\rightarrow\infty}||x_n||=\infty$ or $\{x_n\}$ converges to a critical point of $f$. Moreover, if $f$ has compact sublevels, then only the second alternative happens. 

3) For G3: There is a set $\mathcal{A}\subset \mathbb{R}^m$ of Lebesgue measure $0$, so that if $x_0\notin \mathcal{A}$, and $x_n$ converges to $x_{\infty}$, then $x_{\infty}$ cannot be  a {\bf saddle point} of $f$. Hence, if $\nabla ^2fx_{\infty}$ is invertible, then $x_{\infty}$ must be a local minimum. 

4) If $x_n$ converges to $x_{\infty}$ which is a non-degenerate critical point of $f$, then the rate of convergence is at least {\bf linear}. If moreover, $x_{\infty}$ is a local minimum, then the rate of convergence is quadratic.  

5) (Capture theorem) If $x_{\infty}'$ is a non-degenerate local minimum of $f$, then for initial points $x_0'$ close enough to $x_{\infty}'$, the sequence $\{x_n'\}$  constructed by New Q-Newton's method Backtracking G will converge to $x_{\infty}'$. 

\label{Theorem1}\end{theorem}

The next application concerns Morse functions. We recall that $f$ is Morse if every critical point of its is non-degenerate. That is, if $\nabla f(x^*)=0$ then $\nabla ^2f(x^*)$ is invertible. By transversality theory, Morse functions are dense in the space of functions. We recall that a sublevel of $f$ is a set of the type $\{x\in \mathbb{R}^m:~f(x)\leq a\}$, for a given $a\in \mathbb{R}$. The function $f$ has compact sublevels if all of its sublevels are compact sets. It is not known if New Q-Newton's method satisfies Theorem \ref{Theorem2}.  

\begin{theorem}
Let $f$ be a $C^3$ function, which is Morse. Let $x_0$ be a random initial point, and let $\{x_n\}$ be a sequence constructed by the New Q-Newton's method Backtracking  G. Then either  $\lim _{n\rightarrow\infty}||x_n||=\infty$, or $x_n$ converges to a {\bf critical point} of $f$. Moreover, if $f$ has compact sublevels, then only the second alternative occurs. 
\label{Theorem2}\end{theorem}

Note that when $0<\tau <1$, then part 2 of Theorem \ref{Theorem1} gives a stronger result.   

Next, we discuss the applications to cost functions $f$ which are real analytic or more generally both $f$ and $\nabla f$ satisfy Lojasiewicz gradient inequality (need near critical points of $f$ only). These types of cost functions appear in many realistic applications such as in robotics and Deep Neural Network. Recall that a function $f$ satisfies Lojasiewicz gradient inequality at a point $x^*$ if there is a small neighbourhood $U$ of $x^*$, a constant $0<\mu <1$ and a constant $C>0$ so that for all $x,y\in U$ we have
\begin{eqnarray*}
|f(x)-f(y)|^{\mu}\leq C||\nabla f(x)||. 
\end{eqnarray*}

We will need the following quantity in statements of the next results:

{\bf Definition (Lojasiewicz exponent):} Assume that $f:\mathbb{R}^m\rightarrow \mathbb{R}$ has the Lojasiewicz gradient inequality near its critical points. Then at each critical point $x^*$ of $f$, we define

$\mu (x^*):=\inf \{\mu : $ there is an open neighbourhood $U$ of $x^*$ and a constant $C>0$ so that for all $x,y\in U$ we have $|f(x)-f(y)|^\mu \leq C||\nabla f(x)||\}$.

We say that the gradient $\nabla f$ satisfies Lojasiewicz gradient inequality at a point $x^*$, if  the function $F(x,y)=<\nabla f(x),y>$ $:\mathbb{R}^{2m}\rightarrow \mathbb{R}$ satisfies Lojasiewicz gradient inequality. By Lojasiewicz' theorem, if $f$ is real analytic (and hence $F(x,y)$ is also real analytic), then $f$ and its gradient satisfy Lojasiewicz gradient inequality. Hence, the next theorem can be applied to {\bf quickly} finding roots of systems of (real or complex) analytic equations. Part 3 of its in particularly generalises a result in \cite{truong2021}, which treats the case of finding roots of univariate meromorphic functions (i.e. meromorphic functions in 1 variable). 

\begin{theorem} Assume that $f:\mathbb{R}^m\rightarrow \mathbb{R}$ satisfies the Lojasiewicz gradient inequality. Let $\{x_n\}$ be a sequence constructed by New Q-Newton's Backtracking G. Assume also that $0<\tau \leq 1$. 

1) Assume that $f$ has at most countably many critical points, and $\nabla f$ satisfies the Lojasiewicz gradient inequality. Then either $\lim _{n\rightarrow\infty}||x_n||=\infty$ or $\{x_n\}$ converges to a critical point of $f$. 

2) Assume that for all critical points $x^*$ of $f$, we have $ \mu (x^*)\times (1+\tau ) <1$. For $\tau =1$, assume  moreover that $\nabla f$ also satisfies the Lojasiewicz gradient inequality. Then either $\lim _{n\rightarrow\infty}||x_n||=\infty$ or $\{x_n\}$ converges to a critical point of $f$. 

3) Assume that whenever a subsequence $x_{k_n}$ converges,  we have 
\begin{eqnarray*}
\liminf _{n\rightarrow\infty}\frac{\min _{i\in \Lambda _{k_n}} ||A_{k_n}.e_{i,k}||}{\max _{i\in \Lambda _{k_n}} ||A_{k_n}.e_{i,k_n}||}>0, 
\end{eqnarray*}
where $\Lambda _{k_n}=\{i:~<\nabla f(x_{k_n}),e_{i,k_n}>\not= 0\}$. Moreover, if $\tau =1$, then assume that $\nabla f$  also satisfies the Lojasiewicz gradient inequality. Then either $\lim _{n\rightarrow\infty}||x_n||=\infty$ or $\{x_n\}$ converges to a critical point of $f$. 
\label{Theorem3}\end{theorem}  
The result in part 1 is new only when $\mu =1$ (the remaining case is covered by part 2 of Theorem \ref{Theorem1}). In parts 2 and 3, we do not require that $f$ has only countably many critical points. In part 2), it is allowed that $\sup _{\nabla f(x^*)=0}\mu (x^*)\times (1+\tau )=1$, provided the supremum is not attained at any critical point $x^*$. On the other hand, since in general $\mu (x^*)\geq 1/2$, it follows that if $\mu (x^*)(1+\tau )<1$ then we should have $\tau <1$. The condition in part 3) is proven in \cite{truong-etal} to be satisfied for the case $f:\mathbb{R}^2\rightarrow \mathbb{R}$ is defined as $f=u^2+v^2$, where $u=$ the real part of $g$ and $v=$ the imaginary part of $g$, and $g$ is a univariate holomorphic (more generally, meromorphic) function. The next result is inspired by the mentioned result in \cite{truong2021}, and generalises it for the case of polynomial equations. (Note, however, the proof in \cite{truong2021} establishes the conditions in part 3 of Theorem \ref{Theorem3}, while the proof of the next theorem relies on part 2 of Theorem \ref{Theorem3}.) The result is separated from Theorem \ref{Theorem3} to emphasise its practical usefulness. 

\begin{theorem} Let $P_1,\ldots ,P_N$ be polynomials (in either real or complex variables). Let $f=|P_1|^2+\ldots +|P_N|^2$ be a function from $\mathbb{R}^m\rightarrow \mathbb{R}$. There exists $0<\tau _0<1$ depending only on the dimension and the degrees of the polynomials $P_1,\ldots ,P_m$ so that if $\{x_n\}$ is a sequence constructed by New Q-Newton's method Backtracking G applied to $f$, with $\tau \leq \tau _0$, then we have a bifurcation: Either

i) $\lim _{n\rightarrow\infty}||x_n||=\infty$, 

or 

ii) the whole sequence $\{x_n\}$ converges to a critical point of $f$.

\label{Theorem4}\end{theorem}

In the above theorem, a zero of the system of equations $P_1=\ldots =P_N=0$ is a global minimum of $f$. Theorem \ref{Theorem4} mainly concerns the case where the set of critical points of $f$ is uncountable. In the case where the set of critical points of $f$ is countable one can use also part 1 of Theorem \ref{Theorem3}. For avoidance of saddle point and rate of convergence, one can use Theorem \ref{Theorem1}.

The remaining of this paper is organised as follows. In Section \ref{Section2} we prove the main theoretical results. In the last section we will explore the feasibility of implementation of the algorithms and of the assumptions in the theoretical results.

{\bf Acknowledgments.} The author is partially supported by Young Research Talents grant 300814 from Research Council of Norway.

\section{Proofs of main results}\label{Section2} In this section, we prove Theorems \ref{Theorem1}, \ref{Theorem2}, \ref{Theorem3} and \ref{Theorem4}.

\begin{proof}[Proof of Theorem \ref{Theorem1}] 

0) This is obvious. 

1) We recall that we have the inequality $||A_k.e||\geq minsp(A_k)$, and by definition of the algorithm  $minsp(A_k)\geq \kappa ||\nabla f(x_k)||^{\tau}$ for all $k$ and all vectors $e$ of unit length.  Then we can follow \cite{truong2021} to complete the proof. 

2) Since $0<\tau <1$, we have $||A_k.e_{i,k}||\geq \kappa ||\nabla f(x_k)||^{\tau }$ for all $i=1,\ldots ,m$. Therefore
\begin{eqnarray*}
||w_k||&=&(\sum _{i=1}^m\frac{<\nabla f(x_k),e_{i,k}>^2}{||A_k.e_{i,k}||^2})^{1/2}\\
&\leq&\frac{||\nabla f(x_k)||}{\min _i||A_k.e_{i.k}||}\\
&\leq&||\frac{1}{\kappa}\nabla f(x_k)||^{1-\mu}. 
\end{eqnarray*}
Assume now that $x_{k_n}$ converges to a point $x^*$. Then we know from part 1) that $\nabla f(x^*)=0$. Therefore,
\begin{eqnarray*}
||x_{k_n+1}-x_{k_n}||\leq ||w_{k_n}||\leq \frac{1}{\kappa}||\nabla f(x_{k_n})||\rightarrow 0.
\end{eqnarray*}
Hence, if $d(.,.)$ is the projective metric (see \cite{truong-nguyen1, truong2021}), then $\lim _{k\rightarrow\infty}d(x_{k},x_{k+1})=0$. We then can finish the proof as in \cite{truong-nguyen1, truong2021}.   

3) By the proof of \cite{truong-etal, truongnew}, it suffices to show that near a non-degenerate critical point, the associated dynamical system is $C^1$, and hence Stable-Central manifold theorems can be applied. So, we let $x^*$ be a non-degenerate critical point of $f$, and show that the assignment $x\mapsto H(x)=x-\gamma (x)w(x)$ is $C^1$. 

First of all, as argued in \cite{truong2021}, near a non-degenerate critical point we have $\gamma (x)=1$. Therefore, we need to show that the assignment $x\mapsto x-w(x)$ is $C^1$ near $x^*$. This is clearly the case when $\tau >1$, in that case the map $x\mapsto ||\nabla f(x)||^{\tau}$ is $C^1$. 

Even when $\mu \leq 1$, we still have that $H(x)$ is $C^1$. Indeed, since $x\mapsto ||\nabla f(x)||^{\tau}$ is $C^1$  at non-critical points, it can be checked that $H(x)$ is $C^1$ for non-critical points. Since $x^*$ is the only critical point of $f$ in a small neighbourhood of itself,  it remains to check that the gradient of $H$ at $x^*$ exists, and gives rise to a  $C^1$ function. 

We consider first the case where $e_{1}(x),\ldots ,e_m(x)$'s are $C^1$ functions near $x^*$. Then, 
\begin{eqnarray*}
w(x)=\sum _{i=1}^m\frac{<\nabla f(x),e_i(x)>}{||A(x).e_i(x)||}e_i(x).
\end{eqnarray*}
provided $x\not= x^*$. We then have
\begin{eqnarray*}
\nabla w(x)=\sum _{i=1}^m\frac{\nabla (<\nabla  f(x),e_i(x)>e_{i}(x))}{||A(x).e_i(x)||}+\sum _{i=1}^m\frac{<\nabla f(x),e_i(x)>}{ ||A(x).e_i(x)||^2}e_i(x)\otimes \nabla ||A(x).e_i(x)||.
\end{eqnarray*}
The first sum on the RHS is clearly defined at $x^*$ as well, and is continuous there. Hence, it suffices to show that the same is true for the second sum. We will show that indeed
\begin{eqnarray*}
\lim _{x\rightarrow x^*}\sum _{i=1}^m\frac{<\nabla f(x),e_i(x)>}{\nabla ||A(x).e_i(x)||}e_i(x)=0.
\end{eqnarray*}
To this end, we recall that by arguments in \cite{truong-etal}, near $x^*$ we have $\delta (x)=\delta _0$. Hence $A(x)=\nabla ^2f(x)+\delta _0||\nabla f(x)||^{\tau}$. Therefore, 
\begin{eqnarray*}
\nabla ||A(x).e_i(x)||&=&\nabla ||(\nabla ^2f(x)+\delta _0||\nabla f(x)||^{\tau}).e_{i}(x)||\\
&=&\nabla (\sum _{j=1}^m<(\nabla ^2f(x)+\delta _0||\nabla f(x)||^{\tau}).e_{i}(x),e_j(x)>)^{1/2}\\
&=&\frac{\sum _{j=1}^m<\nabla (\nabla ^2f(x)+\delta _0||\nabla f(x)||^{\tau}).e_{i}(x),e_j(x)>}{2||A(x).e_i(x)||}.
\end{eqnarray*}
 By chain rule,  the norm of $\nabla ||\nabla f(x)||^{\tau }$ is bounded by $||\nabla f(x)||^{\tau -1}$. Therefore,  the norm of 
 \begin{eqnarray*}
 \sum _{i=1}^m\frac{<\nabla f(x),e_i(x)>}{ ||A(x).e_i(x)||^2}e_i(x)\otimes \nabla ||A(x).e_i(x)||
 \end{eqnarray*}
is bounded by $||\nabla f(x)||^{\tau}$, which converges to $0$ as $x$ converges to $x^*$. From this, it is easy to complete the proof. 

Now we consider the remaining case that $e_1(x),\ldots ,e_m(x)$ are eigenvectors of $\nabla ^2f(x)$ near $x^*$. In this case, the functions $x\mapsto e_j(x)$ may not be $C^1$, and hence the above argument cannot be directly applied. However, we can use the integral representations of $pr_{A(x),\pm}$ as in \cite{truong-etal}, and then the above argument can be similarly applied. 

4) We can prove as in \cite{truong-etal}. 

5) This is well-known for methods having the descent property in part 0. 
\end{proof}

\begin{proof}[Proof of Theorem \ref{Theorem2}]
We can use verbatim the proof for the corresponding property for New Q-Newton's method Backtracking in \cite{truong2021}. 
\end{proof}

\begin{proof}[Proof of Theorem \ref{Theorem3}]

1) As in part 2 of Theorem \ref{Theorem1}, it suffices to show that  for a subsequence $x_{n_k}$ converging to a critical point $x^*$: $\lim_{k\rightarrow\infty}||w_{n_k}||=0$. We look at a new function $F:\mathbb{R}^{2m}\rightarrow \mathbb{R}$ given by $F(x,y)=<\nabla f(x),y>$. We have $F(x=x^*,y=0)=0$. Moreover, $\nabla F=(\nabla ^2f(x).y,\nabla f(x))$.  Hence, by assumption, there exists $C>0$ and $\mu <1$, and an open neighborhood $U$ of $x^*$ as well as $\epsilon >0$ so that for all $x\in U$ and all $||y||=1$:
\begin{eqnarray*}
|<\nabla f(x),y>|^{\mu }\leq C(||\nabla ^2f(x).y||+||\nabla f(x)||). 
\end{eqnarray*}

Then, by replacing $C$ with a bigger constant, and substituting $x$ by $x_{n_k}$ and $y$ by $e_{j,k}$ we obtain also the inequality
\begin{eqnarray*}
|<\nabla f(x_{n_k}),e_{j,k}>|^{\mu }\leq C(||\nabla ^2f(x_{n_k}).e_{j,k}+\delta ||\nabla f(x_{n_k})||e_{j,k}||+||\nabla f(x_{n_k})||). 
\end{eqnarray*}
If we choose $\delta$ appropriately from the finite set $\delta _0,\ldots ,\delta _m$, we obtain $A(x_{n_k})$.  

Since $\tau \leq 1$ and $||e_{j,k}||=1$, it follows that $||A(x).e_{j,k}||\geq minsp(A(x))\geq \kappa ||\nabla f(x)||$. Therefore, by replacing $C$ by a bigger constant, we obtain that $|<\nabla f(x_{n_k}),e_{j,k}>|^{\mu }\leq C||A(x).e_{j,k}||$ for every $j$. Therefore, we obtain
\begin{eqnarray*}
||w_{n_k}||\leq C||\nabla f(x_{n_k})||^{1-\mu}\rightarrow 0, 
\end{eqnarray*}
as wanted. 

2) Similar to \cite{truong-etal}, we need to show that if $x^*$ is a critical point of $f$, there exists a constant $C>0$ and $\theta <1$ so that if a subsequence $\{x_{n_k}\}$ converges to $x^*$, then:
\begin{eqnarray*}
\lim _{k\rightarrow \infty}||w_{n_k}||&=&0,\\
<w_{n_k},\nabla f(x_{n_k})>&\geq& C||w_{n_k}||\times |f(x_{n_{k+1}})-f(x_{n_k})|^{\theta}.
\end{eqnarray*}

The equality follows from part 2 of Theorem \ref{Theorem1} (if $0<\tau <1$) and from part 1 above if $\tau =1$. 

Now we prove the inequality. We have
\begin{eqnarray*}
||A_{n_k}.e_{j,n_k}||\geq minsp(A_{n_k})\geq \kappa ||\nabla f(x_{n_k})||^{\tau}. 
\end{eqnarray*}

Then, computed as in \cite{truong2021}, we obtain
\begin{eqnarray*}
<w_{n_k},\nabla f(x_{n_k})>\geq C||w_{n_k}||\times ||\nabla f(x_{n_k})||^{1+\tau}. 
\end{eqnarray*}
From this, the claim follows by the definition of $\mu (x^*)$: we can choose $\theta =\mu (1+\tau )<1$ for some choices of $\mu \geq \mu (x^*)$. 

3) The proof is similar to that in \cite{truong-etal}.

\end{proof}

\begin{proof}[Proof of Theorem \ref{Theorem4}]

The function $f$ is a real polynomial in dimension $m$ and of degree $d\geq 2$ (except the trivial case when f is a constant). Hence,  by the result in \cite{acunto-kurdyka}, we have $\mu (x^*)\leq 1-1/R(m,d)$ with $R(m,d)=d(3d-3)^{m-1}$ for all critical points $x^*$ of $f$. Hence, we can choose $\tau <1/(1-1/R(m,d))-1$, and apply part 2 of Theorem \ref{Theorem3}. 
\end{proof}

\section{Feasibility}\label{Section3} We provide some comments in the feasibility of the assumptions in the main results, as well as the implementation. 

{\bf The assumptions:}

The assumptions that $f$ is Morse or more generally has at most countably many critical points is a local condition and is satisfied generically. 

The assumption that $f$ has Lojasiewicz gradient inequality is satisfied by many interesting cost functions, and relevant to those used in Deep Learning. 

The assumption that $0<\tau \leq 1$: this only concerns the perturbation term $\delta ||\nabla f(x)||^{\tau}$ (or one can reduce to $\delta \min \{1,||\nabla f(x)||^{\tau}\}$) of the Hessian $\nabla ^2f(x)$, and hence is not a real restriction. Experimentally, we found that choosing $\tau >1$ gives similar performance. 

The assumption $\mu (x^*)(1+\tau )<1$ for all critical points $x^*$ is roughly satisfied for example when the multiplicities at the critical points are bounded from above by a finite number, and $\tau <1$ is small enough.  

For the assumption 
\begin{eqnarray*}
\liminf _{k\rightarrow\infty}\frac{\min _{i\in \Lambda _k} ||A_k.e_{i,k}||}{\max _{i\in \Lambda _k} ||A_k.e_{i,k}||}>0,
\end{eqnarray*}
in part 3 of Theorem \ref{Theorem3}: The term on the LHS is related to the "condition number" for the matrix $A_k$. Note that if at each step $k$ one chooses the vectors $e_{1,k},\ldots ,e_{m,k}$ randomly (or if one pertube them randomly), then each $e_{j,k} $ would not be completely orthogonal to the eigenvector(s) of the largest eigenvalue of $A_k$, and hence one expect that the mentioned assumption would hold. Note that actual computations on computers have inherent errors, and hence can contribute to make the process behave as random.

{\bf Implementation:} 

If one wants to compute $minsp(A_k)$, then there are two ways to go. First, one can compute the characteristic polynomial $p_k(t)$ of $A_k$, and then uses one numerical method (e.g. New Q-Newton's method or Backtracking gradient descent) to find if $p_k(t)$ has a root in the ball $\{z\in \mathbb{C}:~|z|<\kappa ||\nabla f(x_k)||^{\tau}\}$. Second, one can use numerical methods for the optimization problem $<Ax,x>^2$ for $x$ on the Riemannian manifold $\{||x||=1\}$. 

If one does not check the condition $minsp(A_k)\geq \kappa ||\nabla f(x_k)||^{\tau}$, but only checks the weaker condition $\det (A_k)\not= 0$, then experimental results (on various types of problems) are similar. An other method, which is both theoretically and practically sound, is to choose $\delta $ randomly. This way, with probability $1$ one has $minsp(A_k)$ "large enough" (while not need to check it or even the weaker condition that $A_k$ is invertible).

\end{document}